%% file: lagrangians_in_twistor_spaces_manuscript.tex
\documentclass[reqno]{amsart}
\usepackage{lineno}
\usepackage{amsfonts,amssymb,amsmath,amsthm}
\usepackage[utf8]{inputenc}
\usepackage{url}
\usepackage{enumerate}
\usepackage{calc}
\usepackage{graphicx}
\usepackage{import}
\usepackage{xcolor}
\usepackage{pifont}
\usepackage[all]{xy}
\usepackage{hyperref}
\hypersetup{colorlinks   = true,
	citecolor    = green,
	linkcolor    = blue}
\usepackage{cleveref}
\usepackage{caption}
\usepackage{svg}

\newtheorem{theorem}{Theorem}[section]
\newtheorem{theoremx}{Theorem}
\newtheorem{lemma}[theorem]{Lemma}

\theoremstyle{definition}
\newtheorem{definition}[theorem]{Definition}
\newtheorem{remark}[theorem]{Remark}

\numberwithin{equation}{section}

\global\long\def\C{\mathbb{C}}

\global\long\def\R{\mathbb{R}}

\global\long\def\F{\mathbb{F}}

\global\long\def\mf#1{\mathfrak{#1}}
\global\long\def\mc#1{\mathcal{#1}}

\global\long\def\.{,\dots ,}

\global\long\def\so{\mathfrak{so}}
\global\long\def\Ad{\mathrm{Ad}}
\global\long\def\ad{\mathrm{ad}}

\global\long\def\so{\mathfrak{so}}

\global\long\def\h{\mathfrak{h}}

\global\long\def\m{\mathfrak{m}}
\global\long\def\n{\mathfrak{n}}
\global\long\def\p{\mathfrak{p}}

\global\long\def\End{\operatorname{End}}
\global\long\def\sd{\mathrm{d}}
\global\long\def\<{\langle}
\global\long\def\>{\rangle}

\author{Reinier Storm}
\address{KU Leuven, Department of Mathematics, Celestijnenlaan 200B -- Box 2400, BE-3001 Leuven, Belgium} 
\email{reinier.storm@kuleuven.be}
\thanks{This work is supported by project 3E160361 of the KU Leuven Research Fund and by the EOS research project 3E180171.}


\begin{document}
	
\title[Lagrangian submanifolds and superminimal surfaces]{A note on Lagrangian submanifolds of twistor spaces and their relation to superminimal surfaces} 

\begin{abstract}
In this paper a bijective correspondence between superminimal surfaces of an oriented Riemannian $4$-manifold and particular Lagrangian submanifolds of the twistor space over the $4$-manifold is proven.
More explicitly, for every superminimal surface a submanifold of the twistor space is constructed which is Lagrangian for all the natural almost Hermitian structures on the twistor space. 
The twistor fibration restricted to the constructed Lagrangian gives a circle bundle over the superminimal surface.
Conversely, if a submanifold of the twistor space is Lagrangian for all the natural almost Hermitian structures,
then the Lagrangian projects to a superminimal surface and is is contained in the Lagrangian constructed from this surface. 
In particular this produces many Lagrangian submanifolds of the twistor spaces $\mathbb{C} P^3$ and $\mathbb{F}_{1,2}(\mathbb{C}^3)$ with respect to both the Kähler structure as well as the nearly Kähler structure. 
Moreover, it is shown that these Lagrangian submanifolds are minimal submanifolds.
\end{abstract}
\maketitle
	
\section{Introduction}
The twistor space of an oriented Riemannian $4$-manifold is inspired by Penrose's twistor program and introduced in \cite{Atiyah1978}, where the authors use it to classify self-dual solutions of the Yang-Mills equations on $S^4$. 
By now many other spaces which also go by the name of twistor spaces have been constructed; see for example \cite{Salamon1982,Alekseevsky1993, Burstall1990, Berard-Bergery1984,OBrian1985}.  
These twistor spaces have, among other things, been used to study all kinds of harmonic maps and relate these to holomorphic maps; see \cite{Eells1983,Eells1985, Salamon1985}.
Particular instances of this are superminimal surfaces in $4$-manifolds.
These are minimal submanifolds and thus yield a harmonic immersion.
These types of surfaces had been extensively studied from different points of view; see \cite{Boruvka1928, Eisenhart1912, Kommerell1905,Kwietniewski1902}. 
Bryant proved in \cite{Bryant1982} that every compact Riemann surface admits a superminimal immersion into $S^4$ using the twistor fibration $\C P^3\to S^4$.
This inspired many studies of minimal surfaces from this twistor bundle point of view; see for example \cite{Friedrich1984, Chi1996, Loo1989, Jensen1989}. 
There are not many results about Lagrangian submanifolds of these twistor spaces. 
In \cite{Castro2001} Lagrangian submanifolds of the Kähler $\C P^3$ are constructed, which contain the ones constructed in this paper.

\subsection{Results}
In \cite{Storm2019} all totally geodesic and all homogeneous Lagrangian submanifolds of $\F_{1,2}(\C^3)$, the nearly Kähler manifold of full flags in $\C ^3$, are classified.
The flag manifold $\F_{1,2}(\C^3)$ is the twistor space of $\C P^2$.
This paper started out as a continuation of the study of Lagrangian submanifolds in $\F_{1,2}(\C^3)$. 
The relation found between superminimal surfaces of $\C P^2$ and Lagrangian submanifold of $\F_{1,2}(\C^3)$ turns out to apply for any twistor space $Z$ of an oriented Riemannian $4$-manifold. 
For this we have to define almost Hermitian structures on the twistor space. 
The twistor spaces come equipped with two natural almost complex structures, which we denote by $J^\pm$; see $\eqref{eq:Jpm}$ below.
In \eqref{eq:metric} we pick a natural family of metrics $g_\lambda$ for $\lambda >0$ on $Z$ making $(Z,g_\lambda,J^\pm)$ an almost Hermitian space.
We will prove the following theorem in \Cref{sec:lag lift}.
\begin{theoremx}\label{thm:Lag and superminimal}
For a superminimal surface $\Sigma\subset M^4$ there exists a submanifold $L_\Sigma \subset Z$ which is a Lagrangian submanifold for all the almost Hermitian structures $(g_\lambda,J^\pm)$ simultaneously.
This Lagrangian $L_\Sigma$ projects under the twistor fibration to $\Sigma$ and the restriction of the twistor fibration to $L_\Sigma$ determines a circle bundle over $\Sigma$.
Conversely, given a submanifold $L\subset Z$ of the twistor space which is Lagrangian with respect to all the almost Hermitian structures $(g_\lambda,J^\pm)$ simultaneously, then the twistor fibration projects $L$ to a surface $\Sigma$ which is superminimal and $L$ is contained in $L_\Sigma$.
\end{theoremx} 

From the classification of totally geodesic Lagrangian submanifolds of $\mathbb{F}_{1,2}(\mathbb{C}^3)$ in \cite{Storm2019} it follows that each of these is congruent under the symmetry group of the nearly Kähler structure of $\mathbb{F}_{1,2}(\mathbb{C}^3)$ to one which projects to a superminimal surface. 
With the correspondence of \Cref{thm:Lag and superminimal} we obtain infinitely many new examples of Lagrangian submanifolds of both the Kähler and nearly Kähler structures on $\mathbb{F}_{1,2}(\mathbb{C}^3)$ and $\C P^3$ by relating them to superminimal surfaces.
Finally, it is shown that the submanifolds $L_\Sigma \subset (Z,g_\lambda)$ are minimal submanifolds.

\section{Twistor space\label{sec:twistor fibration}}
In this section we give a short introduction to the twistor fibration over a $4$-manifold.
Let $(M^4,g)$ be an oriented Riemannian $4$-manifold. 
For a given complex structure $J\in \End(T_xM)$ which is compatible with the metric, i.e. $J\in O(T_x M)$, let $\omega_J(-,-) = g(J-,-) \in \Lambda^2 T_x^*M$. 
The complex structure $J$ is said to be compatible with the orientation if $\omega_J\wedge \omega_J$ is equal to the orientation of $M^4$ at $x$, which is denoted as $\omega_J \wedge \omega_J \gg 0$.

\begin{definition} 
The \emph{twistor bundle} $\pi:Z \to M^4$ of an oriented Riemannian manifold $(M^4,g)$ is the bundle whose fiber over a point $x \in M^4$ consists of all complex structures on the vector space $T_x M^4$ which are compatible with the Riemannian metric and the orientation, i.e. 
\[
\pi^{-1}(x) = \{J\in \End(T_x M^4) :J^2= -1,\ J^*g = g \textrm{ and } \omega_J \wedge \omega_J \gg 0 \}.
\]
The fiber is isomorphic to $SO(4)/U(2) \cong \C P^1$.
\end{definition}
Alternatively, the twistor bundle can be defined as the associated bundle of the principal $SO(4)$-frame bundle $\mc F$ of $(M^4,g)$ by
\begin{equation}\label{eq:ass bundle}
Z \cong \mc F\times_{SO(4)} SO(4)/U(2).
\end{equation}
Here  $SO(4)/U(2)$ is identified with all complex structures $J$ on $\R^4$ such that $J \in SO(4)$ and $\omega_J\wedge \omega_J\gg 0$. 
Elements of $Z$ are equivalence classes $[u,J]$, where $u:\mathbb{R}^4\to T_xM$ is a frame and $J$ is a complex structure on $\mathbb{R}^4$.
The equivalence relation is given by $[u\cdot g,J] \sim [u,g\cdot J]$ for $g \in SO(4)$.
The Levi-Civita connection on $M^4$ induces a horizontal subbundle $T^hZ$ which is complementary to the vertical subbundle $T^vZ$ of the tangent space of $Z$, i.e.
\[
TZ = T^v Z\oplus T^h Z. 
\]
The derivative of $\pi$ restricted to $T^hZ$ yields an isomorphism between $T^h_I Z$ and $T_{\pi(I)} M$ for all $I \in Z$.
Two natural almost complex structures on $Z$ are defined by
\begin{equation}\label{eq:Jpm}
J^{\pm}_I = \pm J_{\C P^1} + I \in \End(T_I Z),
\end{equation}  
where $J_{\C P^1}$ is the natural complex structure on the fiber $\C P^1$ and $I\in Z$ is a complex structure on the vector space $T_{\pi(I)} M^4 \cong T^h_I Z$. 
In \cite{Atiyah1978} it is shown that $J^+$ is integrable if and only if $(M^4,g)$ is anti-self-dual, which is an essential ingredient for the classification of instantons of $S^4$ in \cite{Atiyah1978}. Moreover, $J^+$ is conformally invariant; see \cite{Atiyah1978}.
The almost complex structure $J^-$ is never integrable; see \cite{Salamon1985}.

A third point of view on the twistor space is the taken in \cite{Alekseevsky1993}.
Where the twistor space is naturally identified as a quotient of $\mc F$, namely $Z := \mc F/U(2)$, where $U(2) \subset SO(4)$ is the stabilizer subgroup of some fixed complex structure $J_0$ on $\R^4$.
The natural isomorphism $\mc F/U(2)$ to our previous definition of the twistor space is given by
\begin{equation*}
u\cdot U(2) \mapsto [u,J_0]. 
\end{equation*} 
This map does not depend on the representative $u$ of the coset and yields a bundle isomorphism.
Let $\so(5) = \so(4) \oplus \p$ be the Cartan decomposition. 
Let $\mf u(2)\subset \so(4)$ be the stabilizer algebra of $J_0$.
This yields a decomposition 
\begin{equation*}
\so(5) = \mf u(2) \oplus \n \oplus \p,
\end{equation*} 
where $\n\subset \so(4)$ is the orthogonal complement of $\mf u(2)$ with respect to the Killing form. 
Throughout this work $B$ will denote the Killing form on $\so(5)$. 
We will denote $\m := \n \oplus \p$ and $\h := \mf u(2) \subset \so(5)$.
The advantage of this point of view is that $Z$ comes equipped with a principal $U(2)$-bundle $\mc F \to Z$ together with a Cartan connection
\begin{equation*}
\phi: \mc F \to \so(5) = \h \oplus \m,
\end{equation*}
which is just the Levi-Civita connection on the principal frame bundle combined with the soldering form.
The homogeneous space $SO(5)/U(2)$ is an adjoint orbit of a central element $z\in \mc Z(U(2))$. 
This induces a Kähler-Einstein structure on $SO(5)/U(2)$.
The metric at the identity coset is given by
\begin{equation*}
g_K = -2 B|_\n \oplus -B|_\p.
\end{equation*} 
The corresponding symplectic form $\omega_K$ is the Kirillov-Kostant-Souriau symplectic form and is up to a scaling factor given by $\ad(z)$.
The associated complex structure is above denoted by $J^+$.
The almost complex structure $J^-$ is in this setting given by 
\begin{equation*}
J^-|_\p = J^+|_\p\qquad \mbox{and}\qquad J^-|_\n = - J^+|_\n.
\end{equation*} 

More generally, The model space $SO(5)/U(2)$ comes equipped with a natural family of $U(3)$ structures
\begin{equation}\label{eq:metric}
(g_\lambda,J^\pm),
\end{equation} 
where $g_\lambda = -\frac{1}{\lambda^2}B|_\n \oplus -B|_\p$. 
The nearly Kähler structure on $SO(5)/U(2)$ is given by $(SO(5)/U(2),g_1,J^-)$.
These $U(3)$-structures can be pulled back to $Z$ by the Cartan connection.
Moreover, these structures are parallel with respect to this connection.
In the same way it is possible to induce a natural family of $SU(3)$-structures on $Z$. 
Let $\Upsilon$ denote the natural complex volume form induces from the nearly Kähler structure on $SO(5)/U(2)$.
Then 
\begin{equation}\label{eq:SU(3)-str}
(g_\lambda,J^\pm,\lambda \Upsilon)
\end{equation} 
defines a family of $SU(3)$-structures on $Z$.
% However, this will not be directly relevant for this paper.
The results here do not depend on the particular value of $\lambda$.
Some obvious properties of $g_\lambda$ are that $(Z,g_\lambda,J^\pm)$ is an almost Hermitian manifold and the fibers $\pi^{-1}(x)\subset Z$ are totally geodesic submanifolds.

\section{superminimal surfaces\label{sec:superminimal surfaces}}
Let $\Sigma$ be an oriented $2$-dimensional surface in a $4$-dimensional oriented Riemannian manifold $(M^4,g)$.  
Let $N\Sigma\subset TM^4$ be the normal bundle of $T\Sigma$. 
Let $J_0$ be the complex structure on $TM^4|_\Sigma$ defined by a rotation by $\frac{\pi}{2}$ in $T_x\Sigma$ and a rotation by $\frac{\pi}{2}$ in $N_x\Sigma$. 
This determines a lift $F_0:\Sigma \to Z$ of the inclusion $i:\Sigma\to M^4$:
\[
\xymatrix{
		& Z\ar[d]^\pi
		\\
	\Sigma\ar[r]_i\ar[ur]^{F_0}  & M^4
} 
\]
with $\pi \circ F_0 = i$.
\begin{remark}
For an oriented surface $\Sigma \subset M^4$ there is a natural $U(2)$-reduction of the restricted frame bundle $\mc F|_{\Sigma}$. 
Whenever we mention the group $U(2)$ below we are referring to this $U(2)$-reduction given by the complex structure $F_0$ along $\Sigma$.
\end{remark}

Fix an oriented orthonormal local frame $u = (e_1,e_2,e_3,e_4)$ of $TM^4$ such that $(e_1,e_2)$ is an oriented basis of $T\Sigma$. 
The complex structure $J_0$ is determined by $J_0(e_1) = e_2$ and $J_0(e_3)=e_4$.
In the identification of $Z$ with the associated bundle as in \eqref{eq:ass bundle} the lift $F_0:\Sigma \to Z$ is locally given by
\begin{equation}\label{eq:F local}
	F_0(x) =[u(x),J_0],
\end{equation} 
where we make some slight abuse of notation by denoting the complex structure on $\R^4$ corresponding to $J_0$ in the frame $u$ also by $J_0$.
\begin{definition}\label{def:superminimal}
An oriented surface $\Sigma\subset M^4$ is \emph{superminimal} if the vertical component $(\sd F_0)^v$ of $\sd F_0$ vanishes.  
\end{definition}
\begin{remark}
It is easy to see that a superminimal surface is in particular minimal. An oriented minimal surface $\Sigma\subset M^4$ is superminimal if in addition $F_0:\Sigma \to (Z,J^+)$ is holomorphic with respect to the complex structure of $\Sigma$ induced from the conformal structure; see \cite{Friedrich1984}. 
\end{remark}
The following lemma is quite trivial, but nevertheless it is a useful equivalent condition for superminimal surfaces. 
\begin{lemma}\label{lem:equiv superminimal}
For an oriented surface $\Sigma \subset M^4$ the following are equivalent
\begin{enumerate}[(i)]
\item the surface $\Sigma$ is superminimal,
\item the complex structure $F_0$ on $TM|_\Sigma$ is parallel, i.e. $\mathrm{Hol}(\nabla|_\Sigma) \subset U(2)$, where $U(2)$ is defined with respect to $F_0$. 
\end{enumerate}
\end{lemma}
\begin{proof}
Let $\gamma:I\to \Sigma$ be a curve.
We will use the description of $Z$ as associated bundle $\mc F\times_{SO(4)} SO(4)/U(2)$.
Thus the local formula \eqref{eq:F local} gives
\begin{align*}
 (\sd_{\gamma'(t)}F_0)^v &= \left(\frac{d}{dt} [u(\gamma(t)),J_0]\right)^v = \left(\frac{d}{dt}[u^h(\gamma(t))\cdot g(t),J_0]\right)^v \\
   &=  \left(\frac{d}{dt}[u^h(\gamma(t)),g(t)J_0]\right)^v = \frac{d}{dt}g(t)J_0,
\end{align*} 
where $u^h(\gamma(t))$ is a horizontal lift of $\gamma$ and $g(t) \in SO(4)$ is the parallel translation along $\gamma$ with respect to $\nabla|_\Sigma$ in the frame $u$. 
In the last equality we identified the fiber $\pi^{-1}(\gamma(t))$ with the space of complex structures on $\R^4$ via $u^h(\gamma(t))$.
This implies that $g(t) \in \mathrm{stab}(J_0) \cong U(2)\subset SO(4)$ if and only if $(\sd_{\gamma'(t)} F_0)^v = 0$.
Consequently, the holonomy of the $\nabla|_\Sigma$ is contained in $U(2)$ if and only if $\Sigma$ is superminimal.
\end{proof}
There are a couple of interesting other equivalent formulations of superminimal surfaces. 
For example a surface $\Sigma\subset M^4$ is superminimal if and only if the indicatrix, also known as the curvature ellipse, of $\Sigma$ is a circle centered at zero. 
Or alternatively, $\Sigma$ is superminimal if and only if it is negatively oriented-isoclinic. 
This was proven in \cite{Kwietniewski1902} for $M = \R^4$ and the general case is proven in \cite{Friedrich1984}. 
These two properties of surfaces in a $4$-manifold have been studied in the beginning of the 20th century.
The equivalence of these other characterizations to \Cref{def:superminimal} is proven in \cite{Friedrich1984}. 

\section{A correspondence between Lagrangians and superminimal surfaces\label{sec:lag lift}}
Let $\Sigma \subset M^4$ be a superminimal surface. 
Below we will construct a submanifold $L_\Sigma\subset Z$, which is Lagrangian with respect to all of the almost Hermitian structures $(Z,g_\lambda,J^\pm)$.
The restriction of the twistor fibration $\pi|_{L_\Sigma}:L_{\Sigma}\to \Sigma$ will be a circle bundle and the fibers of $\pi|_{L_{\Sigma}}$ are geodesics in $(Z,g_\lambda)$.
In particular $L_{\Sigma}$ is a $\lambda$-ruled Lagrangian submanifold in the sense of \cite{Lotay2011}.

Let $\Sigma$ be a superminimal surface in $(M^4,g)$ and let $N\Sigma\subset TM^4$ denote the normal bundle of $T\Sigma$. 
Let $u=(e_1,e_2,e_3,e_4)$ be an oriented orthonormal local frame of $M$ such that $(e_1,e_2)$ is an oriented orthonormal basis of $T\Sigma$.
Define a submanifold $L_\Sigma$ of $Z$ as a bundle over $\Sigma$ by
\[
L_{\Sigma} \cap \pi^{-1}(x) := \{J\in Z_x(M^4): J(T_x\Sigma) = N_x\Sigma \}.
\]
Note that $\pi|_{L_{\Sigma}} :L_{\Sigma}\to \Sigma$ is a circle bundle and the fibers can in terms of the frame be expressed as the complex structures determined by
\begin{equation}\label{eq:J_theta}
J_\theta(e_1) = \cos(\theta)e_3 + \sin(\theta)e_4,\quad J_\theta(e_2) = \sin(\theta)e_3 - \cos(\theta)e_4,
\end{equation} 
for $\theta\in S^1$. 
Just as we did with $J_0$ we will also denote the complex structure on $\mathbb{R}^4$ corresponding to $J_\theta$ in the frame $u$ by $J_\theta$.
These complex structures at a tangent space $T_xM^4$ for $x\in \Sigma$ can be depicted as in \Cref{fig:cplx_str}. 
\begin{figure}
\begin{center}
\def\svgscale{0.5}
\def\svgwidth{\columnwidth}
%\includesvg{sphere_almost_cplx_str}
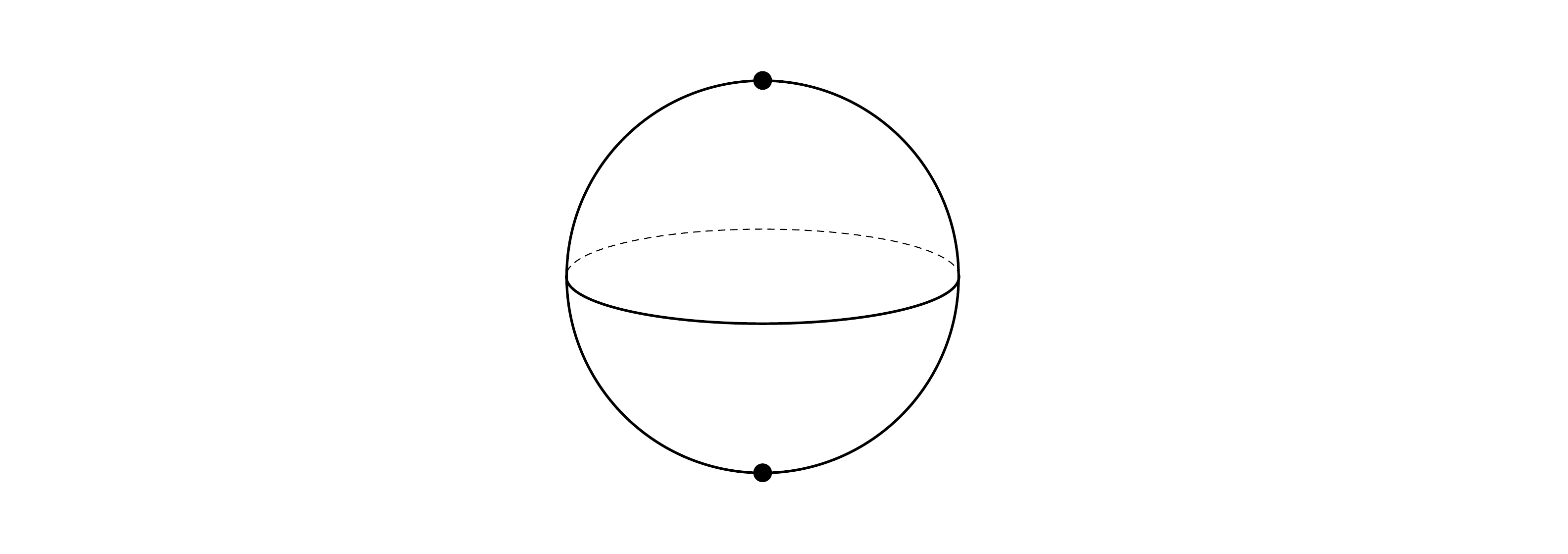
\captionof{figure}{Complex structures of $T_xM^4$ at a point $x\in\Sigma$.\label{fig:cplx_str}}
\end{center}
\end{figure}
\noindent The complex structures $J_\theta$ form the equator of the fiber $\pi^{-1}(x)$ with respect to the poles $\{J_0,-J_0\}$. 

\begin{remark}\label{rem:stabilizer of S1}
The important property of the equator $S^1_x = \{J_\theta\}_{\theta \in S^1}$ of complex structures is that the subalgebra which preserves this equator is exactly the stabilizer algebra of $J_0$, i.e. $\mf u(2)\subset \so(4)$.   
\end{remark}  

\begin{proof}[Proof of \Cref{thm:Lag and superminimal}]
Suppose $\Sigma$ is a superminimal surface. 
Fix an oriented orthonormal local frame $u = (e_1,e_2,e_3,e_4)$ of $TM^4$ such that $(e_1,e_2)$ is an oriented basis of $T\Sigma$. 
First we show that the tangent space of $L_\Sigma$ is compatible with the splitting of $Z$ into the vertical and horizontal subbundles. 
A point of $L_\Sigma$ is in the frame $u$ expressed as $J_\theta$ for some $\theta$.
Note that tangent space of $L_\Sigma$ which is contained in the vertical subbundle of $TZ$ is equal to $T_{J_\theta}S^1_x$, where $\pi(J_\theta) = x \in\Sigma$.
Let $F_\theta:\Sigma\to L_\Sigma\subset Z$ be the local lift given by $F_\theta(y) = [u(y),J_\theta]$.
By \Cref{lem:equiv superminimal} we know that parallel translation in $Z$ along a curve in $\Sigma$ preserves $J_0$. 
Thus by \Cref{rem:stabilizer of S1} it also preserves the subbundle $L_\Sigma \subset Z$. 
From this we obtain that the vertical component of $\sd F_\theta(x)$ is contained in $T_{J_\theta}S^1_x$.  
Clearly $\mbox{im}(\sd F_\theta(x)) \oplus T_{J_\theta} S^1_x = T_{J_\theta} L_\Sigma$ holds.
Consequently, the tangent space of $L_\Sigma$ at $J_\theta$ splits into a vertical and a horizontal subspace, i.e.
\[
T_{J_\theta} L_\Sigma = T^v_{J_\theta} L_\Sigma \oplus  T^h_{J_\theta}L_\Sigma.
\]
Let $X \in T_{J_\theta} L_\Sigma$ and write it as $X = X^v + X^h$, with  $X^v \in T^v_{J_\theta} L_\Sigma$ and $X^h \in T^h_{J_{\theta}}L_\Sigma$ and similarly for $Y\in T_{J_\theta} L_{\Sigma}$. 
We have
\[
g_\lambda(J^{\pm}(X),Y) = \pm g_\lambda(J_{\C P^1}(X^v),Y^v) + g_\lambda(J_{\theta}(X^h),Y^h)= 0, 
\]
where $J_{\C P^1}(X^v)$ is perpendicular to $Y^v$ because $X^v$ and $Y^v$ are linearly dependent and $J_{\theta}(X^h)$ is perpendicular to $Y^h$ because $J_{\theta}$ maps $T\Sigma$ to $N\Sigma$. 
Thus $J^\pm(TL_\Sigma)$ is perpendicular to $TL_\Sigma$ with respect to the metric $g_\lambda$. 
We conclude that $L_\Sigma\subset (Z,g_\lambda,J^\pm)$ is a Lagrangian submanifold with respect to all the almost Hermitian structures.

Conversely, suppose we are given a submanifold $L\subset (Z,g_\lambda)$ which is Lagrangian for both $J^+$ and $J^-$. 
Let $\omega_\pm$ be the Kähler forms of these almost Hermitian structures, i.e. 
\begin{equation*}\label{eq:omega_pm}
\omega_\pm(X,Y) = g_\lambda(J^\pm(X),Y). 
\end{equation*}
Let $\omega_\pm = \omega_\pm^v + \omega_\pm^h$ be the decomposition into their vertical and horizontal parts. 
The vertical and horizontal parts of $\omega_\pm$ satisfy $\omega^v_+ = -\omega^v_-$ and $\omega^h_+ = \omega^h_-$.  
Thus both $2$-forms $\omega^v_+$ and $\omega^h_+$ vanish on $L$. 
The projection of $T_IL$ on the horizontal distribution $T^h_IZ$ is a Lagrangian subspace for every $I\in L$ with respect to the restricted almost Hermitian structure, because $\omega^h_+$ vanish on $L$. 
Thus this projection necessarily has a non-trivial kernel. 
The projection of $T_IL$ onto the vertical distribution $T^v_IZ$ can be at most $1$-dimensional, because $\omega^v_+$ vanishes on $L$. 
Consequently, the projection of $TL$ onto the vertical distribution is equal to its intersection with the vertical distribution and the splitting
\[
T_IL = T^v_IL \oplus T^h_I L
\]
holds. 
In particular $\Sigma:=\pi(L)\subset M^4$ is a $2$-dimensional surface. 
As before we fix some oriented orthonormal local frame $u = (e_1,e_2,e_3,e_4)$ on $M^4$ such that $(e_1,e_2)$ is an oriented basis of $T\Sigma$. 
A complex structure in $L$ is then of the form $J_\theta$ as in \eqref{eq:J_theta} for some $\theta$,
because $L$ is Lagrangian.
A lift of $\pi|_L:L\to \Sigma$ is given by $F_\theta=[u,J_\theta]:\Sigma\to L$. 
This implies $(\sd F_\theta)^v\in T^v_{J_{\theta}}L = T_{J_\theta}S^1_x$ and by \Cref{rem:stabilizer of S1} this implies that the holonomy group of $\nabla|_\Sigma$ is contained in $U(2)$. Therefore, by \Cref{lem:equiv superminimal} we obtain that $\Sigma$ is superminimal.  
\end{proof}

\subsection{Minimal Submanifolds}
For Lagrangian submanifolds in nearly Kähler manifolds it is known that they are minimal submanifolds; see \cite{Gutowski2003}.
Apart from these two homogeneous nearly Kähler twistor spaces the almost Hermitian manifolds $(Z,g_\lambda,J^\pm)$ are never nearly Kähler. 
However, in this section we will show that the Lagrangians constructed in \Cref{sec:lag lift} are minimal submanifolds for all of these spaces $(Z,g_\lambda)$ simultaneously.

In this section we denote the Levi-Civita connection on $Z$ by $\nabla^g$, where the index $\lambda$ is excluded from the notation.
The connection on $Z$ obtained from the Cartan connection $\phi:T\mc F \to \so(5) = \h \oplus \m$ on $\mc F \to Z$ is denoted by $\nabla$; see also \Cref{sec:twistor fibration} and \cite{Alekseevsky1993} for more details.
Furthermore, let $A = \nabla^g - \nabla$.
We will use $\nabla$ to compute the mean curvature.
Let 
\begin{equation*}
B_0 :=\frac{1}{\sqrt{2}} \begin{pmatrix} 1  & 0 & 0 & -1 \\ 0 & 1 & -1  & 0 \\ 0 & 1 & 1 & 0 \\ 1 & 0 & 0 & 1 \end{pmatrix}
\end{equation*} 
and $B_\theta := \exp(\theta J_0) B_0$ for $\theta \in S^1$. 
These matrices are chosen such that $B_\theta  J_0  B_\theta^{-1}= J_\theta$. 
Let $u$ be an adapted frame along $\Sigma$ as before.
We define a lift $\hat{u}: L_\Sigma \to \mc F$ of $\pi:\mc F \to Z$ by
\begin{equation*}
\hat{u}([u(x),J_\theta]) = u(x) \cdot B_\theta.
\end{equation*} 
This is a local lift, because
\begin{equation*}
\pi(\hat{u}([u(x),J_\theta])) = [u(x)\cdot B_\theta, J_0] = [u(x),B_\theta\cdot J_0] = [u(x), J_\theta].
\end{equation*} 
Let us denote $p_{x,\theta} := \hat{u}([u(x),J_\theta])$. 
Let $R_h:\mc F \to \mc F$ denote the right principal bundle action by an element $h\in SO(4)$.
We define an orthonormal frame of $L_\Sigma$ as follows.
Let
\begin{equation*}
\hat{v}_3(p_{x,\theta}) = \frac{\lambda}{\sqrt{12}} \frac{d}{d \theta} R_{B_\theta}(u(x)) \in T_{p_{x,\theta}} \mc F
\end{equation*}   
and for $i=1,2$ we let
\begin{equation*}
\hat{v}_i (p_{x,\theta}) = {R_{B_\theta}}_* u_*e_i \in T_{p_{x,\theta}}\mc F. 
\end{equation*}  
Furthermore, if we put $v_i = \pi_* \hat{v}_i$, then $(v_1,v_2,v_3)$ is an orthonormal frame for $L_\Sigma$ with respect to the metric $g_\lambda$.
The vector $v_3$ is vertical for the bundle $Z \to M$ and the vectors $v_1$ and $v_2$ are horizontal.

The mean curvature will by computed by
\begin{equation*}
\sum_{i=1}^3 (\nabla_{v_i}v_i)^\perp = \sum_{i=1}^3 (\nabla^g_{v_i}v_i)^\perp + (A(v_i)v_i)^\perp.
\end{equation*} 
We will show that $\sum_{i=1}^3 (\nabla_{v_i}v_i)^\perp = \sum_{i=1}^3( A(v_i)v_i)^\perp = 0$ and thus that the mean curvature vector of $L_\Sigma \subset (Z,g_\lambda)$ vanishes. 

First we consider $\sum_{i=1}^3 (\nabla_{v_i}v_i)^\perp$. 
We identify a vector field $X$ on $Z$ with its $U(2)$-equivariant function $\Psi(X):\mc F \to \m$, where $\m$ is defined in \Cref{sec:twistor fibration}.
The covariant derivative $\nabla$ can be expressed in terms of the Cartan connection by
\begin{equation*}
\Psi({\nabla_{v_i}v_i}) = \sd_{\hat{v}_i}\phi_\m(\hat{v}_i) + \phi_\h(\hat{v}_i)\cdot \phi_\m(\hat{v}_i),
\end{equation*} 
where $\cdot $ denotes the restricted adjoint action of $\h$ on $\m$.
First of all we have $\phi(\hat{v}_3) = \frac{\lambda}{\sqrt{12}} \Ad(B_0)^{-1}\cdot J_0 \in \n$, where $J_0$ is seen as an element of $\h \subset \so(5)$.
This is a constant function and $\phi_\h(\hat{v}_3) = 0$, thus we find $\nabla_{v_3}v_3 = 0$.
Furthermore, 
\begin{align*}
\sum_{i=1}^2 \Psi(\nabla_{v_i}v_i)^\perp &= \sum_{i=1}^2 (\sd_{\hat{v}_i} \phi_{\m}( {R_{B_{\theta}}}_* u_* e_i) + \phi_{\h}({R_{B_\theta}}_* u_*e_i)\cdot \phi_{\m}({R_{B_\theta}}_* u_*e_i))^\perp\\
 &= \sum_{i=1}^2 (\Ad(B_\theta)^{-1}\sd_{\hat{v}_i} \phi_{\m}(u_* e_i) + \Ad(B_\theta)^{-1}\phi_{\h}(u_*e_i) \Ad(B_\theta) \cdot \phi_{\m}({R_{B_\theta}}_* u_*e_i))^\perp\\
 &= \left(\Ad(B_\theta)^{-1}\sum_{i=1}^2 \sd_{\hat{v}_i} \phi_{\m}(u_* e_i) + \phi_{\h}(u_*e_i)\cdot \phi_{\m}(u_*e_i)\right)^\perp\\
 &= 0,
\end{align*} 
where $B_\theta \in O(4)$ and in the last equality the sum expresses the mean curvature vector of $\Sigma \subset M$ which is zero, because $\Sigma$ is superminimal.

Next we consider $\sum_{i=1}^3 A(v_i)v_i$. 
The connection form $A$ with respect to the Levi-Civita connection can be expressed in terms of the torsion $T$ of $\nabla$ as
\begin{equation*}
g_\lambda(A(X)Y,Z) = \frac{1}{2}\left( T(X,Y,Z) - T(Y,Z,X) + T(Z,X,Y)\right). 
\end{equation*} 
Thus we find
\begin{align*}
\sum_{i=1}^3 g_\lambda(A(v_i)v_i,Z) &=\sum_{i=1}^3 \frac{1}{2}(T(v_i,v_i,Z) - T(v_i,Z,v_i) + T(Z,v_i,v_i))\\
 &=\sum_{i=1}^3 T(Z,v_i,v_i). 
\end{align*} 
The torsion tensor $T$ is identified with an equivariant function $\hat{T}:\mc F \to \Lambda^2 \m^* \otimes \m$, which is given by
\begin{equation*}
\hat{T} = \kappa_\m - t,
\end{equation*} 
where $\kappa_\m$ is the $\m$-component of the curvature function of $\phi$ and $t$ is the constant function which takes the value $g_\lambda (t(x,y),z) = g_\lambda([x,y],z)$.
The torsion of the Cartan connection is given by
\begin{equation*}
\kappa_\m = \kappa_\n + \kappa_\p = R_\n - [\phi_\n,\phi_\p],
\end{equation*} 
where $\kappa_\n = R_\n$ is the Riemannian curvature tensor of $(M^4,g)$ followed by a projection onto $\n$; see also \cite{Alekseevsky1993}. 

\begin{lemma}\label{lem:computation for A}
If either $X \in \n$ or $X \in \p$, then for all $Z \in \m$ the following hold:
\begin{enumerate}[(i)]
 \item $g_\lambda(t(Z,X),X) = 0$,
 \item  $g_\lambda(R(Z,X)_\n,X) = 0$,
 \item $g_\lambda([\phi_n(Z),\phi_\p(X)] - [\phi_\n(X),\phi_\p(Z)],X)=0$.
\end{enumerate}
\end{lemma}
\begin{proof}
For \emph{(i)} we define $t_\lambda\in \bigotimes^3 \m^*$ by $t_\lambda(X,Y,Z) = g_\lambda([X,Y],Z)$.
Note that $t_1 \in \Lambda^3 \m^*$, because $g_1$ is equal to minus the Killing form of $\so(5)$ and thus invariant under the adjoint action.
Suppose $Z \in \n$ and $X \in \n$, then $t_\lambda(Z,X,X) = \lambda t_1(Z,X,X) = 0.$ 
If $Z \in \p$ and $X\in \n$, then $[Z,X]_\m \in \p$ and thus $t_\lambda(Z,X,X) = 0$.
Similarly, for $Z \in \p$ and $X\in \p$ we have  $[Z,X]_\m \in \n$ and thus $t_\lambda(Z,X,X) = 0$.
Lastly, if $Z \in \n$ and $X\in \p$, then $t_\lambda(Z,X,X) = t_1(Z,X,X) = 0$.
By a similar argument we obtain \emph{(iii)} holds.

The Riemannian curvature tensor is a horizontal tensor on $\mc F$, thus $R(X,Y) = 0$ if either $X\in  \n$ or $Y \in \n$.
From this \emph{(ii)} follows.
\end{proof}

From this lemma we obtain
\begin{equation*}
\sum_{i=1}^3 T(Z,v_i,v_i) = 0,
\end{equation*} 
for all $Z$.
We conclude that the Lagrangian submanifold $L_\Sigma \subset (Z,g_\lambda)$ is a minimal submanifold.
\begin{remark}
A $k$-form $\alpha$ which satisfies $\alpha(\eta)\leq 1$ for every normalized $k$-multivector is called a generalized calibration; see \cite{Gutowski2003}. 
A $k$-dimensional submanifold $L$ which satisfies $\alpha(\eta_x) = 1$ for every $x\in L$ is called calibrated, where $\eta_x$ is a normalized top degree multivector on $T_x L$.
If in addition the form $\alpha$ satisfies $\sd \alpha =0$, then it is called a calibration form.
Calibrated submanifolds of calibrations are automatically volume minimizing within their homology class; see \cite{Harvey1982}.
An interesting phenomenon occurs for calibrated submanifolds of nearly Kähler manifolds.
In \cite{Gutowski2003} it is shown that special Lagrangian submanifolds are automatically minimal and in \cite{Lubbe2014} it is shown that pseudo-holomorphic curves are also minimal.
Furthermore, there are no pseudo-holomorphic surfaces of (strict) nearly Kähler structures. 
Thus every calibrated submanifold of a natural generalized calibration form of a nearly Kähler manifold is minimal.
We have shown that the Lagrangian submanifolds $L_\Sigma$, which we constructed here are also minimal. 
It might be interesting to point out that they are also calibrated by the complex volume form $\lambda \Upsilon$, which was defined in \eqref{eq:SU(3)-str} and is a generalized calibration.
\end{remark}

% \bibliographystyle{abbrvurl}
% \bibliography{../../lagrangian.bib}

\end{document}

%% file: sphere_almost_cplx_str_svg-tex.pdf_tex
%% Creator: Inkscape inkscape 0.92.4, www.inkscape.org
%% PDF/EPS/PS + LaTeX output extension by Johan Engelen, 2010
%% Accompanies image file 'sphere_almost_cplx_str_svg-tex.pdf' (pdf, eps, ps)
%%
%% To include the image in your LaTeX document, write
%%   \input{<filename>.pdf_tex}
%%  instead of
%%   \includegraphics{<filename>.pdf}
%% To scale the image, write
%%   \def\svgwidth{<desired width>}
%%   \input{<filename>.pdf_tex}
%%  instead of
%%   \includegraphics[width=<desired width>]{<filename>.pdf}
%%
%% Images with a different path to the parent latex file can
%% be accessed with the `import' package (which may need to be
%% installed) using
%%   \usepackage{import}
%% in the preamble, and then including the image with
%%   \import{<path to file>}{<filename>.pdf_tex}
%% Alternatively, one can specify
%%   \graphicspath{{<path to file>/}}
%% 
%% For more information, please see info/svg-inkscape on CTAN:
%%   http://tug.ctan.org/tex-archive/info/svg-inkscape
%%
\begingroup%
  \makeatletter%
  \providecommand\color[2][]{%
    \errmessage{(Inkscape) Color is used for the text in Inkscape, but the package 'color.sty' is not loaded}%
    \renewcommand\color[2][]{}%
  }%
  \providecommand\transparent[1]{%
    \errmessage{(Inkscape) Transparency is used (non-zero) for the text in Inkscape, but the package 'transparent.sty' is not loaded}%
    \renewcommand\transparent[1]{}%
  }%
  \providecommand\rotatebox[2]{#2}%
  \newcommand*\fsize{\dimexpr\f@size pt\relax}%
  \newcommand*\lineheight[1]{\fontsize{\fsize}{#1\fsize}\selectfont}%
  \ifx\svgwidth\undefined%
    \setlength{\unitlength}{911.85662132bp}%
    \ifx\svgscale\undefined%
      \relax%
    \else%
      \setlength{\unitlength}{\unitlength * \real{\svgscale}}%
    \fi%
  \else%
    \setlength{\unitlength}{\svgwidth}%
  \fi%
  \global\let\svgwidth\undefined%
  \global\let\svgscale\undefined%
  \makeatother%
  \begin{picture}(1,0.35285152)%
    \lineheight{1}%
    \setlength\tabcolsep{0pt}%
    \put(0,0){\includegraphics[width=\unitlength,page=1]{sphere_almost_cplx_str_svg-tex.pdf}}%
    \put(0.49405221,0.31848505){\color[rgb]{0,0,0}\makebox(0,0)[lt]{\lineheight{1.25}\smash{\begin{tabular}[t]{l}$J_0$\end{tabular}}}}%
    \put(0.62641467,0.17962529){\color[rgb]{0,0,0}\makebox(0,0)[lt]{\lineheight{1.25}\smash{\begin{tabular}[t]{l}$S^1_x = \{J_{\theta}\}_{\theta\in S^1}$\end{tabular}}}}%
    \put(0,0){\includegraphics[width=\unitlength,page=2]{sphere_almost_cplx_str_svg-tex.pdf}}%
    \put(0.48889367,0.01625362){\color[rgb]{0,0,0}\makebox(0,0)[lt]{\lineheight{1.25}\smash{\begin{tabular}[t]{l}$-J_0$\end{tabular}}}}%
  \end{picture}%
\endgroup%